\documentclass[11pt,a4paper]{amsart}
\input{preamble}\mytheoremstyleenglish
\usepackage[alphabetic]{amsrefs}
\usepackage{tikz}
\usetikzlibrary{cd}

\title
[Supersingular reduction of Kummer surfaces in char.\ $2$]
{Supersingular reduction of Kummer surfaces in residue characteristic $2$}

\author{Yuya Matsumoto}
\date{2023/02/19}
\address{\tusaddressfull}
\email{\gmail}
\email{\tusrsmail}
\thanks{This work was supported by JSPS KAKENHI Grant Number 16K17560 and 20K14296.}

\newcommand{\EDP}{\mathrm{EDP}}
\newcommand{\circled}[1]{\textcircled{\raisebox{-0.9pt}{\small #1}}}

\begin{document}

\begin{abstract}
Given an abelian surface $A$, defined over a discrete valuation field and having good reduction, does the attached Kummer surface $\Km(A)$ also have good reduction?
In this paper we give an affirmative answer in the extreme case,
that is, when the abelian surface has supersingular reduction in characteristic $2$.
%
%
\end{abstract}

\maketitle

\section{Introduction}

Let $K$ be a Henselian discrete valuation field of characteristic $0$ 
with valuation ring $\cO$, maximal ideal $\fp$, 
and residue field $k$ assumed to be perfect.
We say that a K3 surface or an abelian surface $X$ over $K$ has \emph{good reduction}
if there exists an algebraic space $\cX$ over $\cO$, proper and smooth, whose generic fiber $\cX \otimes_{\cO} K$ is isomorphic to $X$.
We say that $X$ has \emph{potential good reduction} if there exists a finite extension $K'/K$ such that the base change $X_{K'}$ has good reduction.

The N\'eron--Ogg--Shafarevich criterion (Serre--Tate \cite{Serre--Tate}*{Theorem 1}) shows that good reduction of an abelian variety $A$
is equivalent to the unramifiedness of the Galois action on its $l$-adic cohomology group $\Het^1(A_{\overline{K}}, \bQ_l)$.
We (\cite{Matsumoto:goodreductionK3}*{Theorem 1.1}, \cite{Liedtke--Matsumoto}*{Theorem 1.3}) showed a similar criterion for K3 surfaces (using $\Het^2$), 
assuming the existence of semistable models with certain properties and the residue characteristic being large enough.


In this paper we focus on Kummer surfaces.
For an abelian surface $A$ over a field $F$, the Kummer surface $\Km(A)$ is defined to be the minimal resolution of $A / \set{\pm 1}$.
It is known that $\Km(A)$ is a K3 surface if and only if $\charac F \neq 2$ or $A$ is non-supersingular (Proposition \ref{prop:non-extreme quotient}). 
From this geometric nature of Kummer surfaces,
the following question arises naturally.
\begin{ques}
Let $A$ be an abelian surface over the discrete valuation field $K$, and let $X = \Km(A)$ be the attached Kummer surface,
which is a K3 surface since we are assuming $\charac K = 0$.
Assume $A$ has good reduction, with proper smooth model (N\'eron model) $\cA$ over $\cO$.
\begin{enumerate}
\item Does $X = \Km(A)$ have potential good reduction? \label{item:good reduction implies good reduction}
\item Are smooth proper models of $X$ related to $\cA$ in a geometric way? \label{item:relation between models}
\item Is the special fiber $\cX_k$ isomorphic to $\Km(\cA_k)$? \label{item:special fiber being kummer}
\end{enumerate}
\end{ques}

Let us say that we are in the \emph{extreme case} 
if $\cA_k$ is supersingular of characteristic $2$, 
and in the \emph{non-extreme case}
if otherwise (i.e.\ if $\charac k \neq 2$, or $\charac k = 2$ and $\cA_k$ is non-supersingular).
If we are in the non-extreme cases,
then it is known that the all questions have essentially affirmative answers (Proposition \ref{prop:good reduction of non-extreme Kummer surfaces}):
Roughly, one can blow-up the ``singularity'' of the quotient $\cA / \set{\pm 1}$ of the N\'eron model $\cA$ of $A$ to obtain a smooth proper model of $X = \Km(A)$.

If we are in the extreme case,
then the above construction does not give a smooth proper model of $X = \Km(A)$,
mainly because the singularity of $\cA_k / \set{\pm 1}$ is non-rational.
In this case $\Km(\cA_k)$ is a rational surface, and (\ref{item:special fiber being kummer}) cannot hold.
What we prove in this paper is that (\ref{item:good reduction implies good reduction}) and (\ref{item:relation between models})
 holds even in this extreme case.
\begin{thm} \label{thm:good reduction:bis}
Let $K$ and $k$ be as above.
Suppose that $\charac K = 0$ and $\charac k = 2$, 
that an abelian surface $A$ over $K$ has good reduction with N\'eron model $\cA$,
and that the special fiber $\cA_k$ is supersingular.
Then $X = \Km(A)$ has potential good reduction.
\end{thm}
We overview the method of our proof and its consequences.
We give (in Section \ref{subsec:explicit computation}) 
a proper birational morphism $\map{\phi}{\tilde{\cX}}{\cA/\set{\pm 1}}$,
defined as the composite of explicit normalized blow-ups at singular points of the special fiber,
from a proper model $\tilde{\cX}$ of $X$ that is strictly semistable in the broad sense (Definition \ref{def:strictly semistable in the broad sense}).
Then, according to Theorem \ref{thm:sufficient condition for good reduction}, we can run a relative MMP 
and obtain a birational map $\map[\rationalto]{f}{\tilde{\cX}}{\cX}$ to a smooth proper model $\cX$ of $X$,
which achieves good reduction.

The special fiber $\cX_k$ of $\cX$ is not isomorphic, nor even birational, to the special fiber $\cA_k / \set{\pm 1}$ of $\cA / \set{\pm 1}$, since the latter is a rational surface.
Hence the birational map $f \circ \phi^{-1}$ (resp.\ its inverse $\phi \circ f^{-1}$) contracts $\cA_k / \set{\pm 1}$ (resp.\ $\cX_k$).
We may understand that the ``K3-ness'' of the special fiber $\cA / \set{\pm 1}$ lies in the local ring at the singular point of $\cA_k / \set{\pm 1}$,
not on the smooth complement of the point in $\cA_k / \set{\pm 1}$.

Moreover, our construction shows that the special fiber $\cX_k$ is
birational to the quotient of a certain rational surface $A'$ by a certain group scheme of length $2$ (see Remark \ref{rem:inseparable kummer}).
We can modify $A'$ and make it an ``abelian-like'' surface,
by which we mean a surface sharing numerical properties with abelian surfaces.
Thus $\cX_k$ can be viewed as an example of an inseparable analogue of Kummer surfaces.
We will study this subject in another paper.

The proof of Theorem \ref{thm:sufficient condition for good reduction} is essentially given in \cite{Matsumoto:goodreductionK3}*{Section 3}, 
except that we use Takamatsu--Yoshikawa's MMP \cite{Takamatsu--Yoshikawa:mixed3fold}*{Theorem 1.1} (which is applicable in the case of residue characteristic $p = 2$)
in place of Kawamata's (which requires $p \geq 5$).

\section{Preliminaries}

\subsection{Some elliptic double points}

In this subsection, we work over an algebraically closed field $k$ of characteristic $2$.

\begin{defn} \label{def:equations of EDPs}
For $(m,n) = (4, 4), (2, 8), (2, 6)$ and $r = 0, 1$, we say that
$k[[x,y,z]] / (z^2 + r z x^{m/2} y^{n/2} + x^{m+1} + y^{n+1})$
is $\EDP_{m, n}^{(r)}$.
\end{defn}

\begin{rem} \label{rem:EDPs}
The following properties of $\EDP_{m, n}^{(r)}$ hold.
\begin{itemize}
\item
They are all elliptic double points. 
\item
The exceptional divisor of the minimal resolution of $\EDP_{m,n}^{(r)}$
is independent of $r$ and consists of: 
\begin{itemize}
\item one cuspidal rational curve of arithmetic genus $1$ for $\EDP_{2,6}^{(r)}$;
\item smooth rational curves, with the dual graph shown in Figure \ref{fig:dual graph of EDP},
for $\EDP_{2,8}^{(r)}$ and $\EDP_{4,4}^{(r)}$.
\end{itemize}
In the figures, a vertex with $-m$ inside denotes a smooth rational curve of self-intersection $-m$, 
and the absence of a number denotes self-intersection $-2$.
The symbols $\circled{19}_{0}$ and $\circled{4}_{0,1}^{1}$ are due to Wagreich \cite{Wagreich:elliptic singularities}*{Theorem 3.8}.
\item 
$\EDP_{m, n}^{(1)}$ is a $\bZ/2\bZ$-quotient singularity by \cite{Artin:wild2}*{Theorem}, and 
$\EDP_{m, n}^{(0)}$ is an $\alpha_2$-quotient singularity by \cite{Matsumoto:k3alphap}*{Theorem 3.8}.
\end{itemize}
\end{rem}

\begin{rem} \label{rem:another form of EDP2,8}
We have alternative equations for $\EDP_{2, 8}^{(r)}$ and $\EDP_{4, 4}^{(r)}$.
\begin{itemize}
\item 
$\EDP_{2, 8}^{(r)}$ is isomorphic to 
$k[[x,y,z]] / (z^2 + r z x^2 (x + y^2) + x (x + y^2)^2 + x^4 y)$.
Indeed, starting from this equation and letting $x' = x + y^2$, $z' = z + x' y$, $x' = (x'' + r y^5) (1 + r z)$,
we obtain $z'^2 u_1 + r z' x'' y^4 u_2 + x''^3 u_3 + y^9 u_4 = 0$ for some units $u_i$, 
and by replacing $x'',y,z'$ with suitable unit multiples we obtain the equation in Definition \ref{def:equations of EDPs}.
\item 
$\EDP_{4, 4}^{(r)}$ is also isomorphic to 
$k[[x,y,z]] / (z^2 + r z x^2 y^2 + x^4 y + x y^4)$.
Indeed, starting from the equation of Definition \ref{def:equations of EDPs},
letting $\zeta$ be a primitive $5$-th root of $1$ and 
$v = x + \zeta y$ and $w = \zeta x + y$,
we have $x^5 + y^5 = (\zeta^2 + \zeta^3) (v^4 w + v w^4)$, and so on.
\end{itemize}
\end{rem}

The EDPs with $(m,n) = (4,4), (2,8)$ are relevant to us because of the following.
(The remaining one with $(m,n) = (2,6)$ is used to describe possible degenerations.)

\begin{prop}[Katsura \cite{Katsura:Kummer 2}] \label{prop:supersingular quotient}
Suppose $A$ is a supersingular abelian surface in characteristic $2$.
If $A$ is superspecial (resp.\ non-superspecial), then the singularity of $A/\set{\pm 1}$ 
consists of one point of type $\EDP_{4, 4}^{(1)}$ (resp.\ $\EDP_{2, 8}^{(1)}$).
\end{prop}
Here, a supersingular abelian surface is called \emph{superspecial}
if it is isomorphic (not only isogenous) to the product of two supersingular elliptic curves.
\begin{proof}
In the superspecial case, this is
\cite{Katsura:Kummer 2}*{Proposition 8}.
In the non-superspecial case, 
\cite{Katsura:Kummer 2}*{Lemma 12}
gives an equation similar to the one given in Remark \ref{rem:another form of EDP2,8}.
\end{proof}

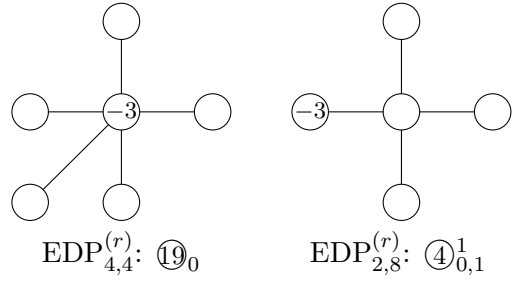
\begin{figure}
\begin{tabular}{ccc}
\begin{tikzpicture}[scale = 0.6]
	\draw (2,2) circle [radius=0.4];
	\draw (0,2) circle [radius=0.4];
	\draw (0,0) circle [radius=0.4];
	\draw (2,0) circle [radius=0.4];
	\draw (2,4) circle [radius=0.4];
	\draw (4,2) circle [radius=0.4];
	\draw (0.4,2)--(1.6,2) (2.4,2)--(3.6,2);
	\draw (2,0.4)--(2,1.6) (2,2.4)--(2,3.6);
	\draw (0.28,0.28)--(1.72,1.72);
	\draw (2,2) node {\footnotesize $-3$};
\end{tikzpicture}
& &
\begin{tikzpicture}[scale = 0.6]
	\draw (2,2) circle [radius=0.4];
	\draw (0,2) circle [radius=0.4];
	\draw (2,0) circle [radius=0.4];
	\draw (2,4) circle [radius=0.4];
	\draw (4,2) circle [radius=0.4];
	\draw (0.4,2)--(1.6,2) (2.4,2)--(3.6,2);
	\draw (2,0.4)--(2,1.6) (2,2.4)--(2,3.6);
	\draw (0,2) node {\footnotesize $-3$};
\end{tikzpicture}
\\
$\EDP_{4,4}^{(r)}$: $\circled{19}_{0}$  & \qquad &
$\EDP_{2,8}^{(r)}$: $\circled{4}_{0,1}^{1}$
\end{tabular}
\caption{Dual graphs of minimal resolutions of EDPs}
\label{fig:dual graph of EDP}
\end{figure}

\subsection{Quotients by group schemes of length $2$ in mixed characteristic} \label{subsec:quotient 2}

We recall Tate--Oort's classification \cite{Tate--Oort:groupschemes}*{Theorem 2} of finite group schemes $G$ of length $2$.
For our purposes it suffices to consider $G$ over a base ring $\cO$ that is a PID (possibly a field).
Then, such group schemes are of the form $G_{\alpha, \beta}$, 
defined as follows,
with parameters $\alpha, \beta \in \cO$ satisfying $\alpha \beta = 2$
(which are denoted by $a, b$ in \cite{Tate--Oort:groupschemes}).

The underlying scheme is $G_{\alpha, \beta} = \Spec R$, $R = \cO[X] / (X^2 - \alpha X)$.
The comultiplication map is given by $\mapandmapsto{R}{R \otimes_{\cO} R}{X}{X \otimes 1 + 1 \otimes X - \beta X \otimes X}$.
Actions of $G_{\alpha, \beta}$ on affine schemes $\Spec S$ correspond,
via $\mapandmapsto{S}{R \otimes_{\cO} S}{s}{1 \otimes s + X \otimes \delta(s)}$,
to $\cO$-linear maps $\map{\delta}{S}{S}$ satisfying the conditions
$\delta(s t) = \delta(s) t + s \delta(t) + \alpha \delta(s) \delta(t)$ and $\delta(\delta(s)) = - \beta \delta(s)$.
In particular, $\id + \alpha \delta$ is an automorphism of the $\cO$-algebra $S$ of order dividing $2$.

For a unit $e \in \cO^*$, the group schemes $G_{\alpha, \beta}$ and $G_{e \alpha, e^{-1} \beta}$ are isomorphic under the obvious maps.

\begin{example} \label{ex:group schemes of length 2}
Three important cases of $G_{\alpha, \beta}$ are the following.
\begin{itemize}
\item If $(\alpha, \beta) = (1, 2)$, then $G_{\alpha, \beta}$ is isomorphic to the constant group scheme $\bZ/2\bZ$.
The map $\delta$ is equal to $g - \id$, where $g$ is (the action on $S$ of) the nontrivial element of $\setZ/2\setZ$.
\item If $(\alpha, \beta) = (2, 1)$, then $G_{\alpha, \beta}$ is isomorphic to $\mu_2$.
Actions on $\Spec S$ also correspond to $\bZ/2\bZ$-gradings $S = \bigoplus_{i \in \bZ/2\bZ} S_i$ on the $\cO$-algebra $S$, and then $\delta$ is equal to $(-1)$ times the projection to $S_1$.
If moreover $2 = 0$ in $\cO$, then the maps $\delta$ are precisely the derivations of multiplicative type.
\item If $(\alpha, \beta) = (0, 0)$ (which implies $2 = 0$ in $\cO$), then $G_{\alpha, \beta}$ is isomorphic to $\alpha_2$.
The maps $\delta$ are precisely the derivations of additive type.
\end{itemize}
If $\cO$ is a field of characteristic $2$, then any $G$ is isomorphic to exactly one of above.

If $2$ is invertible in $\cO$, then any $G$ is isomorphic to both $\bZ/2\bZ$ and $\mu_2$.
\end{example}

The \emph{fixed locus} $\Fix(G)$ of an action on $\Spec S$ is the closed subscheme, or the closed subset,
corresponding to the ideal of $S$ generated by $\Image(\delta)$.
In other words, its complement is the largest open subscheme where $G$ acts freely.

\begin{prop} \label{prop:quotient singularity}
Let $\cO$ be a local PID (possibly a field) with maximal ideal $\idealp$ and residue field $k$.
Let $G = G_{\alpha, \beta}$ be a group scheme of length $2$ over $\cO$, acting on $\cO[[u, v]]$,
and assume $\Fix(G) \supset (u = v = 0)$ (as subsets of $\Spec \cO[[u, v]]$)
and $\Fix(G_k \actson k[[u,v]]) = (u = v = 0)$ (as subsets of $\Spec k[[u, v]]$).
Then, the elements
\begin{alignat*}{2}
x &:= u (u + \alpha \delta(u)), &
a &:= \beta u + \delta(u), \\
y &:= v (v + \alpha \delta(v)), &
b &:= \beta v + \delta(v), \\
z &:= \beta u v + \delta(u) v + u \delta(v), 
\end{alignat*}
are $G$-invariant, and 
the invariant subalgebra $\cO[[u, v]]^G$ is equal to 
$\cO[[x, y, z]] / (F)$, 
$F = z^2 - \alpha a b z + a^2 y + b^2 x - \beta^2 x y$.

Suppose $\charac (\Frac \cO) = 0$ and $\charac k = 2$.
Let $\tau$ be the Tyurina number of $k[[x, y, z]] / (F)$.
If $G \otimes_{\cO} k = \alpha_2$, then 
$\tau = 2 \deg \Fix(G \actson k[[u,v]]) = 2 \deg \Fix(G \actson (\Frac \cO) \otimes_{\cO} \cO[[u, v]])$.
If $G \otimes_{\cO} k = \mu_2$, then 
$\tau = 2$ and $\deg \Fix(G \actson k[[u,v]]) = \deg \Fix(G \actson (\Frac \cO) \otimes_{\cO} \cO[[u, v]]) = 1$.
\end{prop}

If $G_{\alpha, \beta} = G_{1, 2} = \bZ/2\bZ$,
then the formulas simplify to 
$x = u g(u)$, $y = v g(v)$, $z = u g(v) + g(u) v$, 
$a = u + g(u)$, $b = v + g(v)$,
where $g$ is the non-trivial element of the group.

\begin{proof}
It is straightforward to check that $x,y,z,a,b$ are invariant.
Letting $z' := z + \alpha \delta(u) \delta(v)$, we observe $z + z' = \alpha a b$ and $z z' = a^2 y + b^2 x - \beta^2 x y$,
which implies $F(x,y,z) = 0$.
Thus it remains to show that $x, y, z$ generates $\cO[[u, v]]^G$.
Since this can be checked modulo the maximal ideals, we may assume that $\cO = k$ is a field.
Then $G$ is isomorphic to one in Example \ref{ex:group schemes of length 2}
(i.e.\ $\bZ/2\bZ$, $\mu_2$, or $\alpha_2$).

Note that in any case the ideal generated by $\Image(\delta)$ is generated by two elements $\delta(u), \delta(v)$.

Suppose $G = \mu_2$ or $G = \alpha_2$.
Then $x = u^2, y = v^2$. 
It follows from \cite{Matsumoto:k3alphap}*{Theorem 3.8} that $k[[u, v]]^G$ has a $k[[x, y]]$-basis of the form $1,w$.
Write $w = c_{00} + c_{10} u + c_{01} v + c_{11} u v$ with $c_{ij} \in k[[u^2, v^2]] = k[[x,y]]$.
Write $z = d + e w$ with $d, e \in k[[x,y]]$.
It suffices to show $e \in k[[x,y]]^*$.
We have $z^2 = d^2 + e^2 w^2 = d^2 + e^2 c_{00}^2 + e^2 c_{10}^2 x + e^2 c_{01}^2 y + e^2 c_{11}^2 x y$.
We also have $z^2 = b^2 x + a^2 y - \beta^2 x y$.
It suffices to show that ${a^2,b^2,\beta^2}$ have no common divisor. 
By the assumption on the fixed locus, the ideal $(\delta(u), \delta(v)) = (a, b)$ is supported on the closed point,
and so is the ideal $(a^2, b^2)$.

Now suppose $G = \bZ/2\bZ$.
In this case $x = N(u)$, $y = N(v)$, and $z = \Tr(u g(v))$,
and it follows from Lemma \ref{lem:wild Z/2Z-action} that 
$x, y, z$ generates $\cO[[u, v]]^G$.

Latter assertion.
Suppose $G_k = \alpha_2$. Then we have $\deg \Fix(G \actson k[[u,v]]) = \deg k[[u,v]]/(a,b)$,
and $\tau = \dim_k k[[x,y,z]] / (a^2, b^2, z^2) = 2 \dim_k k[[x,y]] / (a^2, b^2) = 2 \dim_k k[[u,v]] / (a, b)$
since $x = u^2$ and $y = v^2$ in this case.

Suppose $G_k = \mu_2$.
In this case we can linearize the action and then $\delta(u) = u$ and $\delta(v) = v$, hence $\deg \Fix(G \actson k[[u,v]]) = 1$.
\end{proof}

\begin{lem} \label{lem:wild Z/2Z-action} 
Suppose $k$ is a field of characteristic $2$,
and $G = \bZ/2\bZ = \set{1, g}$ acts on $k[[u, v]]$, freely outside the closed point.
Let $x = N(u) = u g(u)$, $y = N(v) = v g(v)$, and $z = \Tr(u g(v)) = u g(v) + g(u) v$.
Then $k[[u, v]]^G = k[[x,y,z]] / (z^2 + a b z + a^2 y + b^2 x)$,
where $a = \Tr(u) = u + g(u)$ and $b = \Tr(v) = v + g(v)$.
\end{lem}
\begin{proof}
It is straightforward to check $z^2 + a b z + a^2 y + b^2 x = 0$
(cf.\ the first paragraph of the proof of Proposition \ref{prop:quotient singularity}).
It remains to show that $x,y,z$ generate the maximal ideal $\idealn'$ of $k[[u, v]]^G$.

By \cite{Artin:wild2}*{Theorem},
there exist local coordinates $u', v'$ with the following properties:
letting $x' = N(u')$, $y' = N(v')$, $z' = \Tr(u' g(v'))$, 
and $R_0 = k[[x', y']]$,
\begin{itemize}
\item
$k[[u, v]]$ has an $R_0$-basis $1, u', v', u' g(v')$ and 
\item
$k[[u, v]]^G$ has an $R_0$-basis $1, z'$.
\end{itemize}
Write 
\[
\vectorii{u}{v} = 
\begin{pmatrix}
p_{00} & p_{10} & p_{01} & p_{11} \\
q_{00} & q_{10} & q_{01} & q_{11} 
\end{pmatrix}
\vectoriv{1}{u'}{v'}{u' g(v')},
\]
where $p_{ij}, q_{ij} \in R_0$,
$p_{00}, q_{00} \in (x', y')$, and 
$\det \begin{pmatrix} p_{10} & p_{01} \\ q_{10} & q_{01} \end{pmatrix} \in R_0^*$.
Then a straightforward computation yields 
\[
\vectoriii{x}{y}{z} \equiv
\begin{pmatrix}
p_{10}^2 & p_{01}^2 & p_{10} p_{01} \\
q_{10}^2 & q_{01}^2 & q_{10} q_{01} \\
0 & 0 & p_{10} q_{01} + p_{01} q_{10}
\end{pmatrix}
\vectoriii{x'}{y'}{z'}
\pmod{\idealn'^2}.
\]
Since the matrix is invertible, $x, y, z$ generate $\idealn'/\idealn'^2$.
\end{proof}

\section{Non-extreme cases}

The good reduction problem of Kummer surfaces in the non-extreme cases are solved 
by Ito (see \cite{Matsumoto:SIP}*{Section 4}) if $\charac k \neq 2$,
and by Lazda--Skorobogatov \cite{Lazda--Skorobogatov:reduction of kummer} if $\charac k = 2$ and $\cA_k$ is non-supersingular.
For the readers' convenience, we include a sketch of the proof for these cases.

\subsection{Singularities of Kummer surfaces in the non-extreme cases}

We already mentioned the result for the extreme case in Proposition \ref{prop:supersingular quotient}.
In the non-extreme cases, we have the following.

\begin{prop} \label{prop:non-extreme quotient}
Let $A$ be an abelian surface over a field $F$.
If $\charac F \neq 2$, or $\charac F = 2$ and $A$ is non-supersingular,
then the minimal resolution of $A/\set{\pm 1}$ is a K3 surface.
Moreover, if all points of $A[2]$ are $F$-rational, 
then $A/\set{\pm 1}$ has only RDPs as singularities, and the number and the type of RDPs are
\begin{itemize}
\item $16 A_1$ if $\charac F \neq 2$,
\item $ 4 D_4^1$ if $\charac F = 2$ and $\prank(A) = 2$,
\item $ 2 D_8^2$ if $\charac F = 2$ and $\prank(A) = 1$.
\end{itemize}
\end{prop}
For the assertion on quotient singularities, see \cite{Katsura:Kummer 2}*{Proposition 3} for a proof
(for the determination of the coindices, see \cite{Artin:wild2}*{Examples}).

\subsection{Good reduction of Kummer surfaces in the non-extreme cases}

Let $K, \cO, k$ be as in Introduction.

\begin{prop} \label{prop:quotient and fibers}
Suppose $\cA \to \Spec \cO$ is an abelian scheme over $\cO$.
Then the fibers of $\cA / \set{\pm 1}$ are isomorphic to the quotients of the fibers.
\end{prop}
\begin{proof}
See \cite{Lazda--Skorobogatov:reduction of kummer}*{Proposition 4.6}.
\end{proof}

\begin{prop} \label{prop:flat blowup of RDP}
Suppose $\cX \to \Spec \cO$ is a flat morphism of finite type, with both fibers $\cX_k$ and $\cX_K$ normal surfaces.
Suppose $\cZ \subset \cX$ is a section (i.e. the composite $\cZ \to \cX \to \Spec \cO$ is an isomorphism)
such that $\cZ_k$ and $\cZ_K$ are RDPs of the fibers.
Then $\Bl_{\cZ}(\cX)_K \cong \Bl_{\cZ_K}(\cX_K)$ and $\Bl_{\cZ}(\cX)_k \cong \Bl_{\cZ_k}(\cX_k)$.
\end{prop}

\begin{proof}
See \cite{Lazda--Skorobogatov:reduction of kummer}*{Proposition 4.1} or
\cite{Overkamp:Kummer}*{Lemma 3.10}.
\end{proof}

\begin{prop}
\label{prop:good reduction of non-extreme Kummer surfaces}
Let $K$ and $k$ be as in the introduction (with $\charac K = 0$).
Suppose an abelian surface $A$ over $K$ has good reduction with N\'eron model $\cA$.
Assume either $\charac k \neq 2$, 
or $\charac k = 2$ and $\cA_k$ is non-supersingular.
Then $X = \Km(A)$ has potential good reduction.
\end{prop}
\begin{proof}
Let $\cY := \cA / \set{\pm 1}$.
By Proposition \ref{prop:quotient and fibers}, we have $\cY_K = A/\set{\pm 1}$ and $\cY_k = A_k/\set{\pm 1}$.
By replacing $K$ with a finite extension, we may assume that $\cA_K[2]$ and $\cA_k[2]$ consist respectively of $K$-rational points
and $k$-rational points.
Then, by Proposition \ref{prop:non-extreme quotient},
$\Sing(\cY_K)$ is $16 A_1$ and 
$\Sing(\cY_k)$ is one of $16 A_1$, $4 D_4^1$, or $2 D_8^2$.

Choose an RDP of $\cY_K$, take its closure $\cZ$ in $\cY$, and take the blow-up $\cY^{(1)} = \Bl_{\cZ} \cY$.
Then $\cZ_k$ should be a singular point of $\cY_k$, hence one of the RDPs.
By Proposition \ref{prop:flat blowup of RDP}, each fiber of $\cY^{(1)}$ is the blow-up at an RDP.
Repeating this $16$ times, we obtain a scheme $\cY^{(16)}$ whose fibers are $\Km(\cA_K)$ and $\Km(\cA_k)$.
\end{proof}

\section{Proof of main theorem}

Let $\cO$ be as before (i.e.\ a Henselian discrete valuation ring with residue field $k$).
In this section, $k$ is of characteristic $2$, and the maximal ideal of $\cO$ is denoted by $\idealp = (\varpi)$.

We first show in Section \ref{subsec:proof:preliminary} that, 
in order to prove good reduction of a K3 surface $X$,
it suffices to construct a model of $X$ satisfying a kind of semistability.
Then we give an explicit construction of such a model in Sections \ref{subsec:blow-up}--\ref{subsec:explicit computation}.
We summarize the proof of Theorem \ref{thm:good reduction:bis} in Section \ref{subsec:end of proof}. 

\subsection{Strict semistability in the broad sense} \label{subsec:proof:preliminary}

We need the following variant of the strict semistability.

\begin{defn} \label{def:strictly semistable in the broad sense}
We say that 
a flat scheme $\cX$ of finite type over $\cO$
is \emph{strictly semistable in the broad sense}
if it satisfies the following conditions.
\begin{itemize}
\item The generic fiber $X$ of $\cX$ is smooth. 
\item Every irreducible component of $\cX_k$ is smooth.
\item For each closed point $P$ of the special fiber $\cX_k$, 
$\cO_{\cX,P}$ is \'etale over $\cO[x_1, \dots, x_n] / (x_1 \dots x_r - \epsilon)$
for some $1 \leq r \leq n$ and some $\epsilon \in \idealp$.
\end{itemize}
\end{defn}
Recall that $\cX$ is \emph{strictly semistable}
if we can take $\epsilon$ to be a uniformizer of $\cO$.

The existence of such a model is quite useful for the good reduction problem:

\begin{thm} \label{thm:sufficient condition for good reduction}
Let $X$ be a K3 surface over $K$.
Suppose that the Galois representation $\Het^2(X_{\overline{K}}, \bQ_l)$ is unramified for some auxiliary prime $l \neq p$ after replacing $K$ with a finite extension,
and that $X$ admits a proper model $\cX$ over $\cO$ that is strictly semistable in the broad sense.
Then $X$ has potential good reduction.
\end{thm}

Here, a representation of $\Gal(\overline{K}/K)$ is \emph{unramified}
if the inertia subgroup of $\Gal(\overline{K}/K)$ acts trivially.

\begin{proof}
By replacing $K$ with a finite extension (which does not affect the semistability in the broad sense), 
we may assume that $\Het^2(X_{\overline{K}}, \bQ_l)$ is unramified.

By an argument similar to \cite{Saito:logsmooth}*{Lemma 1.7}, 
$\cX$ is log smooth over $\cO$.
By \cite{Saito:logsmooth}*{Theorem 1.8},
there exist a finite extension $\cO'$ of $\cO$ and 
a morphism $\cX' \to \cX \otimes_{\cO} \cO'$ that is isomorphic above the interior of $\cX$
such that $\cX'$ is proper and strictly semistable over $\cO'$.
Here the interior of $\cX$ is the complement of the union of intersections of two or more components of $\cX_k$
(hence the interior contains the generic fiber).
Thus we may assume that $\cX$ is strictly semistable.

Then we can apply the arguments given in \cite{Matsumoto:goodreductionK3}*{Section 3},
with the use of 
Kawamata's MMP replaced with Takamatsu--Yoshikawa's \cite{Takamatsu--Yoshikawa:mixed3fold}*{Theorem 1.1},
which, unlike Kawamata's, has no restriction on the residue characteristic.
\end{proof}

If $X = \Km(A)$ is the Kummer surface associated to an abelian surface $A$ with good reduction,
then, by the N\'eron--Ogg--Shafarevich criterion, the Galois representation $\Het^1(A_{\overline{K}}, \bQ_l)$ is unramified, and then by the isomorphism
\[ \Het^2(\Km(A)_{\overline{K}}, \bQ_l) \cong \bigoplus_{A[2]} \bQ_l(-1) \oplus \bigwedge^2 \Het^1(A_{\overline{K}}, \bQ_l), \]
$\Het^2(\Km(A)_{\overline{K}}, \bQ_l)$ is also unramified after replacing $K$ with a finite extension.
Thus, our goal is to construct a strictly semistable model in the broad sense of the Kummer surface $X = \Km(A)$.

\subsection{Overview of the blow-up construction} \label{subsec:blow-up}

The main idea is to perform a (weighted) blow-up with respect to suitable coordinates.

Suppose $\cY$ is a smooth affine scheme over $\cO$ of relative dimension $2$
and $G = G_{\alpha, \beta}$ is a group scheme of length $2$ (see Section \ref{subsec:quotient 2}).
Suppose $G$ acts on $\cY$, with the quotient morphism denoted by $\map{\pi}{\cY}{\cY / G = \cX}$,
 and the following properties hold.
\begin{itemize}
\item $\Fix(G_k \actson \cY_k) = \set{y}$, 
and the quotient singularity $\pi(y) \in \cX_k$ is either $\EDP_{4,4}^{(r)}$ or $\EDP_{2,8}^{(r)}$.
\item All points of $\Fix(G_K \actson \cY_K)$ specializes to $y$. 
\end{itemize}
Let $\cY' \to \cY$ be the weighted blow-up specified below, with center supported on $y$. 
The $G$-action extends to $\cY'$. Let $\cX' := \cY'/G$.
We list the properties of $\cY'$ and $\cX'$ to be established.
The construction and the proof will be given in Section \ref{subsec:explicit computation}.

\begin{claim} \label{claim:description of X'}
The following holds.
\begin{enumerate}
\item $\cX'$ is a normalized (weighted) blow-up of $\cX$.
\item $\cX'_k$ consists of two components: the strict transform $Z$ of $\cX_k$ and the exceptional divisor $E$.
\item The $G$-action on $\cY' \setminus \pi^{-1}(Z)$ 
induces an action of $G' = G_{\alpha',\beta'}$ with isolated fixed locus, and the quotient is $\cX' \setminus Z$.
\label{claim:item:X':action}
\item The exceptional divisor $E$ is a normal projective surface and satisfies $H^1(\cO_E) = 0$ and $K_E = 0$. 
In particular, if $\Sing(E)$ consists only of RDPs, then $E$ is an RDP K3 surface. 
\item The image of $\Sing(\cX'_K)$ by the specialization map is contained in $\Sing(E) \setminus Z$.
\end{enumerate}
\end{claim}
Here, an RDP K3 surface is a proper surface with only rational double point singularities (if any)
whose minimal resolution is a K3 surface.

\begin{claim} \label{claim:description of Sing}
After replacing $k$ with a finite separable extension, 
we have the following description of $\Sing(Z)$ and $\Sing(E)$.
\begin{enumerate}
\item
If $y$ is $\EDP_{4,4}^{(r)}$, then there are $P_1, \dots, P_5 \in Z \cap E$ and the following hold.
\begin{enumerate}
\item $\Sing(Z) = \set{P_1, \dots, P_5}$, all RDPs of type $A_1$.
\item $\Sing(E) \cap Z = \set{P_1, \dots, P_5}$, all RDPs of type $A_1$.
\item 
$\Sing(E) \setminus Z$ is one of 
$16 A_1$, $4 D_4^0$, $2 D_8^0$, $2 E_8^0$, $1 D_{16}^0$, or $1 \EDP_{2,8}^{(0)}$.
\label{claim:item:44:same type} 
\end{enumerate}
\item
If $y$ is $\EDP_{2,8}^{(r)}$, then there are $Q_1, \dots, Q_4 \in Z \cap E$ and the following hold.
\begin{enumerate}
\item $\Sing(Z) = \set{Q_1, \dots, Q_4}$, $Q_1$ is a quotient singularity of type ${(1,1)}/{3}$, and all others are RDPs of type $A_1$.
\item $\Sing(E) \cap Z = \set{Q_1, \dots, Q_4}$, $Q_1$ is an RDP of type $A_2$, and all others are RDPs of type $A_1$.
\item Every element of $\Sing(E) \setminus Z$ is one of the following:
$A_1$, $D_4^0$, $D_8^0$, $D_{12}^0$, $D_{16}^0$, $E_8^0$, or $\EDP_{2,6}^{(0)}$.
\item If all singularities of $\Sing(E) \setminus Z$ are of the same type, then $\Sing(E) \setminus Z$ is one of 
$16 A_1$, $4 D_4^0$, $2 D_8^0$, $2 E_8^0$, or $1 D_{16}^0$.
\label{claim:item:28:same type} 
\end{enumerate}
\end{enumerate}
\end{claim}

Let $\cX'' \to \cX'$ be the blow-up specified below (with center supported on $\Sing(Z)$).

\begin{claim}
$\cX''$ is strictly semistable in the broad sense outside $\Sing(E) \setminus Z$.
\end{claim}

%

\begin{rem}
For our purpose (Theorem \ref{thm:good reduction:bis}),
what we need directly are resolutions of 
EDPs of type $\EDP_{4,4}^{(r)}$ and $\EDP_{2,8}^{(r)}$ with $r = 1$.
However we also consider the case $r = 0$ because 
we need the resolution of $\EDP_{2,8}^{(0)}$ for the proof of the case of $\EDP_{4,4}^{(1)}$.
The arguments for $r = 0$ and $r = 1$ are parallel.
\end{rem}
\begin{rem} \label{rem:inseparable kummer}
Our construction also shows that the (affine) surface $E \setminus Z$ is the quotient of $\pi^{-1}(E \setminus Z)$ by $G' \in \set{\mu_2, \alpha_2}$.
Although we will not show it in this paper, 
this can be extended to a quotient morphism $A' \to E$ by $G'$ from a proper non-normal surface $A'$ that satisfies
$\omega_{A'} \cong \cO_{A'}$ and $h^i(\cO_{A'}) = 1, 2, 1$ for $i = 0, 1, 2$.
We say that such $A'$ is an abelian-like surface and $E$ is an inseparable analogue of Kummer surface.
Details will be given in another paper.
\end{rem}

\subsection{Explicit computation of the blow-up} \label{subsec:explicit computation}
\subsubsection{Case of $\EDP_{4,4}^{(r)}$} \label{subsubsec:case 4,4}

First we consider the case where $y \in \cY_k/G$ is $\EDP_{4,4}^{(r)}$.
Let $u,v \in \cO_{\cY,y}$ be local coordinates satisfying 
$(u = v = 0) \subset \Fix(G \actson \cY)$.
Let $x,y,z,a,b \in \cO_{\cY,y}^G$ as in Proposition \ref{prop:quotient singularity}.
Then, by the proposition, the maximal ideal of $\cO_{\cY,y}^G$ is equal to $(x,y,z) + \idealp$,
and we have $F = 0$,
where \[ F = z^2 - \alpha z a b + a^2 y + b^2 x - \beta^2 x y. \]
Since $(u = v = 0) \subset \Fix(G)$, we have $a, b \in (x, y, z)$ (not only $a, b \in (x,y,z) + \idealp$). 
By the description of the blow-up of $\EDP_{4,4}^{(r)}$ at the origin,
we have $\overline{a}, \overline{b} \in \idealm^2$ (where $\idealm = (x,y,z)$)
and moreover 
\[
\overline{a^2 y + b^2 x} \equiv c_{50} x^5 + c_{41} x^4 y + c_{14} x y^4 + c_{05} y^5
\pmod{((x,y)^6 + (z))} 
\]
for some $c_{ij} \in k$.
Here the overline denotes the reduction modulo $\idealp$.
By a coordinate change (on $u, v$), and after replacing $k$ with a finite separable extension,
we may assume $\overline{a} \equiv x^2$, $\overline{b} \equiv y^2 \pmod{((x, y)^3 + (z))}$,
and furthermore
\begin{align*}
a &\equiv x^2 + a_{10} x + a_{01} y, \\
b &\equiv y^2 + b_{10} x + b_{01} y \pmod{((x, y)^3 + (z))}
\end{align*}
with $a_{ij},b_{ij} \in \idealp$.
After replacing $\cO$ with a finite extension if necessary,
we take $\epsilon \in \idealp$ satisfying the following condition:
\[
\frac{a_{10}}{\epsilon^2}, 
\frac{a_{01}}{\epsilon^2}, 
\frac{b_{10}}{\epsilon^2}, 
\frac{b_{01}}{\epsilon^2}, 
\frac{\beta^2}{\epsilon^6} \in \cO,
\quad 
\text{and moreover at least one} \in \cO^*.
\] 
Let $\cY' \to \cY$ be the blow-up at the ideal $(u, v, \epsilon)$.
Then $\cX' := \cY' / G$ is the normalized weighted blow-up at $(x, y, z, \epsilon)$ with respect to the weight $(2, 2, 5, 1)$.

We have $\cX' = \cX'_{\epsilon} \cup \cX'_{x} \cup \cX'_{y}$, where $\cX'_{?} = \Spec A_{?}$ are the affine open subschemes of the blow-up $\cX'$ with the obvious meaning.
We first consider $\cX'_{\epsilon} = \Spec A_{\epsilon}$, whose special fiber is equal to $E \setminus Z$. 
The $\cO$-algebra $A_{\epsilon}$ is \'etale over the subalgebra generated by 
\[
x' := \frac{x}{\epsilon^2}, \quad
y' := \frac{y}{\epsilon^2}, \quad
z' := \frac{z}{\epsilon^5},
\]
subject to $F' = 0$, 
where 
$F'$ is naively ``$\epsilon^{-10} F$'', i.e., 
\[ F' = z'^2 - \epsilon^3 \alpha a' b' z' + a'^2 y' + b'^2 x' - (\epsilon^{-6} \beta^2) x' y' \]
where $a' := \epsilon^{-4} a$ and $b' := \epsilon^{-4} b$.
Let $a_{ij}' = \overline{\epsilon^{-2} a_{ij}} \in k$ and 
$b_{ij}' = \overline{\epsilon^{-2} b_{ij}} \in k$. 
Here, again, the overline denotes the reduction modulo $\idealp$.
We have $A_{\epsilon} \otimes_{\cO} k = k[x',y',z'] / (\overline{F'})$,
$\overline{F'} = z'^2 + a'^2 y' + b'^2 x' - \overline{\epsilon^{-6} \beta^2} x' y' $,
$a' = x'^2 + a'_{10} x' + a'_{01} y'$,
$b' = y'^2 + b'_{10} x' + b'_{01} y'$.

Suppose $\epsilon^{-6} \beta^2 \in \cO^*$, hence its image $\overline{\epsilon^{-6} \beta^2} \in k$ is nonzero.
Then since $\overline{F'}_{x' y'} = \overline{\epsilon^{-6} \beta^2} \neq 0$, 
every singularity of $E \setminus Z$ is of type $A_1$,
and considering the degrees of $\overline{F'}_{x'}$ and $\overline{F'}_{y'}$ we conclude that 
the singularity of $E \setminus Z$ is $16 A_1$.
%

Suppose $\epsilon^{-6} \beta^2 \in \idealp$.
Then we have 
\begin{align*}
\overline{a'} &= x'^2 + a'_{10} x' + a'_{01} y', \\
\overline{b'} &= y'^2 + b'_{10} x' + b'_{01} y' 
\end{align*}
with $a'_{ij},b'_{ij} \in k$, not all $0$ by the definition of $\epsilon$.
Then 
\[
\overline{F'} = z'^2 + x'^4 y' + x' y'^4 + b_{10}'^2 x'^3 + a_{10}'^2 x'^2 y' + b_{01}'^2 x' y'^2 + a_{01}'^2 y'^3.
\]
To describe the singularities, we consider 
$c := (b'_{10} : a'_{10} : b'_{01} : a'_{01}) \in \bP^3$
and the following subsets of $\bP^3$:
\begin{align*}
Q &= \set{(t_0 : t_1 : t_2 : t_3) \mid t_1 t_2 - t_0 t_3 = 0}, \\
\Gamma_1 &= \set{(t_0 : t_1 : t_2 : t_3) \mid t_1 t_2 - t_0 t_3 = t_1^2 - t_0 t_2 = t_2^2 - t_1 t_3 = 0}, \\
\Gamma_2 &= \set{(t_0 : t_1 : t_2 : t_3) \mid t_1 t_2 - t_0 t_3 = t_2^2 - t_0 t_1 = t_1^2 - t_2 t_3 = 0}.
\end{align*}
$Q$ is the image of the Segre embedding of $\bP^1 \times \bP^1$ to $\bP^3$, and 
$\Gamma_1$ and $\Gamma_2$ are twisted cubic curves contained in $Q$.
Then it is straightforward to check that the singularity of $E \setminus Z$ is 
as shown in Figure \ref{fig:singularities 4,4} in the following sense:
a configuration of singularities (e.g.\ $2 D_8^0$) occurs 
if and only if $c$ belongs to the corresponding locus ($Q$)
but not to its subspaces ($\Gamma_1$, $\Gamma_2$, and $\Gamma_1 \cap \Gamma_2$).
We also note that 
$\Gamma_1 \cap \Gamma_2 = \Gamma_1(\setF_4) = \Gamma_2(\setF_4)$.
\begin{figure}
%
\[
\begin{tikzcd}[arrows=-]
& 4 D_4^0 \ar[d]                       & & & \bP^3 \ar[d] & \\
& 2 D_8^0 \ar[dl] \ar[dr]              & & & Q = \bP^1 \times \bP^1 \ar[dl] \ar[dr] & \\
2 E_8^0 \ar[dr] & & 1 D_{16}^0 \ar[dl]   & \Gamma_1 \ar[dr] & & \Gamma_2 \ar[dl] \\
& 1 \EDP_{2,8}^{(0)}                   & & & \Gamma_1 \cap \Gamma_2
\end{tikzcd}
\]
\caption{$\Sing(E) \setminus Z$ in the case of $\EDP_{4,4}^{(r)}$, assuming $\epsilon^{-6} \beta^2 \in \idealp$}
\label{fig:singularities 4,4}
\end{figure}
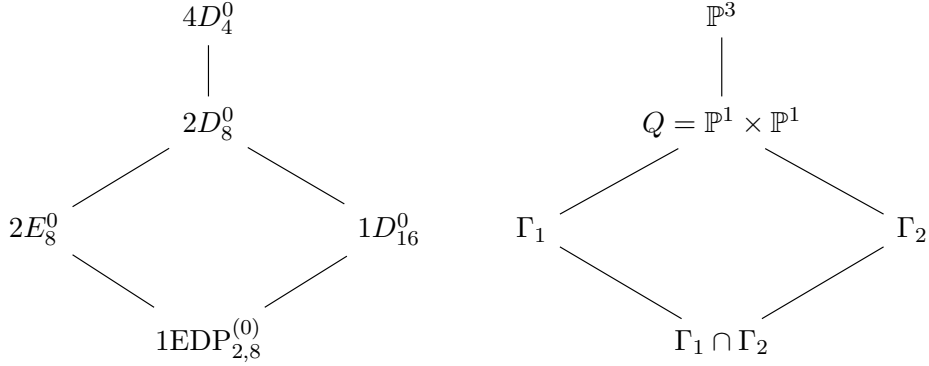

Consider the assertion on the action (Claim \ref{claim:description of X'}(\ref{claim:item:X':action})).
The action of $G = G_{\alpha, \beta}$ on $\cY$ corresponds to the map $\map{\delta}{\cO_{\cY}}{\cO_{\cY}}$ as in Section \ref{subsec:quotient 2}.
Since $\delta(u) = a$ has weight $\geq 4 = 3 + \wt(u)$ 
and $\delta(v) = b$ has weight $\geq 4 = 3 + \wt(v)$ (i.e.\ $a, b \in (u, v)^4$),
it extends to $\map{\delta}{\cO_{\cY'}}{\cO_{\cY'}(-3 E_{\cY})}$,
where $E_{\cY} \subset \cY'$ is the exceptional divisor.
On $\cY' \setminus \pi^{-1}(Z)$, the subsheaf $\cO_{\cY' \setminus \pi^{-1}(Z)}(-3 E_{\cY}) \subset \cO_{\cY' \setminus \pi^{-1}(Z)}$ is equal to $\epsilon^3 \cO_{\cY' \setminus \pi^{-1}(Z)}$,
hence $\delta' := \epsilon^{-3} \delta$ defines a morphism $\cO_{\cY' \setminus \pi^{-1}(Z)} \to \cO_{\cY' \setminus \pi^{-1}(Z)}$
and it corresponds to an action of $G' = G_{\alpha', \beta'}$,
$(\alpha', \beta') = (\epsilon^3 \alpha, \epsilon^{-3} \beta)$.
Since $\delta'(\epsilon^{-1} u) = \epsilon^{-4} a = a'$ and $\delta'(\epsilon^{-1} v) = b'$,
the fixed locus of the action of $G'$ on $\cY'_{\epsilon}$ is $(\overline{a'} = \overline{b'} = 0)$.
This descrpition also shows that the image under the specialization map of $\Sing(X_K)$ is equal to $\Sing(E) \setminus Z$.

Next we consider singularities of $Z$ and $E$ contained in $Z \cap E$, which is contained in $\cX'_{x} \cup \cX'_{y}$.
By symmetry, it suffices to consider $\cX'_{x} = \Spec A_x$.
Define elements of $A_x$ by
\[
y' = \frac{y}{x}, \quad
z' = \frac{z}{x^2}, \quad
z'' = \frac{\epsilon z}{x^{3}}, \quad
q = \frac{\epsilon^2}{x}, \quad
a' = \frac{a}{x^{2}}, \quad
b' = \frac{b}{x^{2}}.
\]
Also let 
\[
s = \frac{z^2}{x^5} = \frac{z'^2}{x} 
= \alpha a' b' z' x  - (b'^2 + a'^2 y' - q^3  (\epsilon^{-6} \beta^2) y').\]
Then the $\cO$-algebra $A_x$ is \'etale over the subalgebra generated by $x,y',z',z'',q$,
subject to
\[
\vectoriii
{s}
{z' \quad z''}
{x \quad \epsilon \quad q}
\in V_2 := \biggset{
\vectoriii{T_0^2}{T_0 T_1 \quad T_0 T_2}{T_1^2 \quad T_1 T_2 \quad T_2^2}
},
\]
where $V_2 \subset \bA^6$ is the cone over the image of the second Veronese map $\bP^2 \to \bP^5$.
From the above description,
it is straightforward to check the following assertions:
\begin{itemize}
\item $\cX_x \cap Z = (\varpi = q = z'' = 0)$,
\item $\cX_x \cap E = (\varpi = x = z' = 0)$,
\item $\cX_x \cap \Sing(Z) \cap E = \cX_x \cap \Sing(E) \cap Z = \cX_x \cap E \cap Z \cap (s = 0)$,
\item $\Sing(Z) \cap E = \Sing(E) \cap Z$ consists of $5$ points $\set{P_1, \dots, P_5}$ on which $(x : y) \in \bP^1(\bF_4)$,
\item each $P_i$ is an RDP of type $A_1$ in both $Z$ and $E$,
\end{itemize}
Define $\cX'' \to \cX'$ to be the blow-up at the ideal $(s, z'', x, z', \epsilon, q)$ over $\cX'_x$, and similar over $\cX'_y$. 
Then $\cX''$ is strictly semistable in the broad sense above each $P_i$.
Each exceptional divisor of $\cX'' \to \cX'$ is $\bP^2$, and
the local equations at the triple points are, for example, $\frac{z'}{s} \frac{z''}{s} s - \epsilon$.
%

By construction, $E$ is a hypersurface of degree $10 = 5 + 2 + 2 + 1$ in $\bP(5,2,2,1)$,
hence we obtain $H^1(\cO_E) = 0$ and $K_E = 0$.

$E \setminus Z$ is the quotient of $E_{\cY} \setminus \pi^{-1}(Z)$ by $G'_k = G_{\overline{\alpha'}, \overline{\beta'}} = G_{0, \overline{\beta'}}$, which is isomorphic to either $\mu_2$ or $\alpha_2$
(respectively if $\overline{\beta'} \neq 0$ or $\overline{\beta'} = 0$).

\subsubsection{Case of $\EDP_{2,8}^{(r)}$} \label{subsubsec:case 2,8}

Next we consider the case where $y$ is $\EDP_{2,8}^{(r)}$.

Let $u,v,x,y,z,a,b,F$ as in the previous case.
%
According to the equation given in Definition \ref{def:equations of EDPs},
we may assume 
\begin{align*}
\overline{a} &\equiv y^4 \pmod{(x, y^5, z)}, \\
\overline{b} &\equiv x   \pmod{(x^2, x y, y^3, z)}. 
\end{align*}
We can  assume moreover (by coordinate changes like $u' = u + \epsilon_1 v$ and $v' = v + \epsilon_2 y$ for $\epsilon_1, \epsilon_2 \in \idealp$) that 
\begin{align*}
a &\equiv y^4 + a_3 y^3 + a_2 y^2 + a_1 y, \\
b &\equiv x + b_1 y \pmod{(x^2, x y, y^5, z)},
\end{align*}
$a_1, a_2, a_3, b_1 \in \idealp$.
We take $\epsilon \in \idealp$ satisfying the following condition (again, after replacing $\cO$ if necessary):
\[
\frac{a_i}{\epsilon^{8-2i}}, 
\frac{b_i}{\epsilon^{6-2i}},
\frac{\beta^2}{\epsilon^{10}}\in \cO,
\quad 
\text{and moreover at least one} \in \cO^*.
\]
Let $\cY' \to \cY$ be the weighted blow-up at the ideal $(u, v, \epsilon)$ with respect to the weight $(3, 1, 1)$.
Then $\cX' := \cY' / G$ is the normalized weighted blow-up at $(x, y, z, \epsilon)$ with respect to the weight $(6, 2, 9, 1)$.

We have $\cX' = \cX'_{\epsilon} \cup \cX'_{x} \cup \cX'_{y}$.
We first consider $\cX'_{\epsilon} = \Spec A_{\epsilon}$, whose special fiber is equal to $E \setminus Z$. 
The $\cO$-algebra $A_{\epsilon}$ is \'etale over the subalgebra generated by 
\[
x' := \frac{x}{\epsilon^6}, \quad
y' := \frac{y}{\epsilon^2}, \quad
z' := \frac{z}{\epsilon^9},
\]
subject to $F' = 0$, 
where \[ F' = z'^2 - \epsilon^5 a' b' z' + a'^2 y' + b'^2 x' - (\epsilon^{-10} \beta^2) x' y' \]
 is naively ``$\epsilon^{-18} F$'',
where $a' := \epsilon^{-8} a$ and $b' := \epsilon^{-6} b$.
Let $a_{i}' = \overline{\epsilon^{-(8-2i)} a_{i}} \in k$ and 
$b_{i}' = \overline{\epsilon^{-(6-2i)} b_{i}} \in k$. 
We have $A_{\epsilon} \otimes k = k[x',y',z'] / (\overline{F'})$,
$\overline{F'} = z'^2 + a'^2 y' + b'^2 x' - \overline{\epsilon^{-10} \beta^2} x' y' $,
$a' = y'^4 + a'_3 y'^3 + a'_2 y'^2 + a'_1 y'^1$,
$b' = x' + b'_1 y'$.

If $\beta^2 \epsilon^{-10} \in \cO^*$ then, as in the case of $\EDP_{4,4}^{(r)}$, 
we observe that $\Sing(E) \setminus Z = 16 A_1$.

Suppose $\epsilon^{-10} \beta^2 \in \idealp$.
Then we have 
\begin{align*}
\overline{a'} &= y'^4 + a'_3 y'^3 + a'_2 y'^2 + a'_1 y', \\
\overline{b'} &= x' + b'_1 y' 
\end{align*}
with $a'_i,b'_i \in k$, not all $0$ by the definition of $\epsilon$.
Then 
\[
\overline{F'} = z'^2 + x'^3 + b_1'^2 x' y'^2 + a_1'^2 y'^3 + a_2'^2 y'^5 + a_3'^2 y'^7 + y'^9. 
\]
Then the singularity at the origin is 
as in Figure \ref{fig:singularities 2,8}.

\begin{figure}
%
\[
\begin{tikzcd}[arrows=-]
D_4^0 \ar[d] &                         & () \ar[d] & \\
D_8^0 \ar[d] \ar[dr] &                 & (a'_1 = 0)  \ar[d] \ar[dr] & \\
D_{12}^0 \ar[d] \ar[dr] & E_8^0 \ar[d] & (a'_1 = a'_2 = 0) \ar[d] \ar[dr] & (a'_1 = b'_1 = 0) \ar[d] \\
D_{16}^0 & \EDP_{2,6}^{(0)}            & (a'_1 = a'_2 = a'_3 = 0) & (a'_1 = a'_2 = b'_1 = 0) 
\end{tikzcd}
\]
\caption{Singularity of $E$ at the origin in the case of $\EDP_{2,8}^{(r)}$, assuming $\epsilon^{-10} \beta^2 \in \idealp$}
\label{fig:singularities 2,8}
\end{figure}
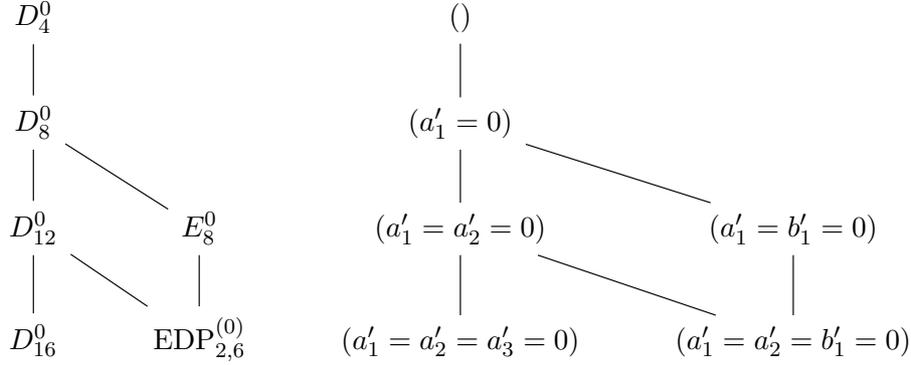
Again we have $\dim k[x',y',z']/I_{\tau} = 32$, where $I_{\tau} = (a'^2, b'^2, z'^2)$ is the Tyurina ideal.
If all singularities are of the same type,
then we conclude that $\Sing(E) \setminus Z$ is one of $4 D_4^0, 2 D_8^0, 2 E_8^0, 1 D_{16}^0$.
In particular, $D_{12}^0$ and $\EDP_{2,6}^{(0)}$ do not appear under this assumption.

The assertion on the extension of the action of $G$ is checked similarly to the case of $\EDP_{4,4}^{(r)}$.
Since $\delta(u) = a$ has weight $\geq 8 = 5 + \wt(u)$
and $\delta(v) = b$ has weight $\geq 6 = 5 + \wt(v)$,
the morphism $\map{\delta}{\cO_{\cY}}{\cO_{\cY}}$ extends to $\map{\delta}{\cO_{\cY'}}{\cO_{\cY'}(-5 E)}$.
We argue similarly (with $3$ replaced with $5$).

Next we consider singularities of $Z$ and $E$ contained in $Z \cap E$, which is contained in $\cX'_{x} \cup \cX'_{y}$.
We check $\cX'_{x} = \Spec A_x$ and omit $\cX'_{y}$ (which does not give any more singularities).
Define elements of $A_x$ by
\begin{gather*}
t_i = \frac{  y^i \epsilon^{6-2i}}{x} \; (i = 0,1,2,3), \quad 
w_i = \frac{z y^i \epsilon^{3-2i}}{x^2} \; (i = 0,1), \quad 
v   = \frac{z y^2}{x^2}, \quad 
z'  = \frac{z}{x}, \\
a'  = \frac{a}{x}, \quad 
b'  = \frac{b}{x}, \quad
s   = \frac{z^2}{x^3} 
= \alpha z' a' b' - (b'^2 + \frac{y a'^2}{x} - \frac{\beta^2}{\epsilon^{10}} t_1 t_0).
\end{gather*}
Then the $\cO$-algebra $A_x$ is \'etale over the subalgebra generated by \[ x,y,t_0,t_1,t_2,t_3,w_0,w_1,v,z', \]
subject to
\begin{gather*}
\rank \begin{pmatrix}
w_0 & \epsilon^2 & t_0 & t_1 & t_2 \\
w_1 & y          & t_1 & t_2 & t_3 \\
\end{pmatrix} \leq 1,
\rank \begin{pmatrix}
w_1 & \epsilon & t_0          & t_1          & t_2 \\
v   & y        & \epsilon t_1 & \epsilon t_2 & \epsilon t_3 \\
\end{pmatrix} \leq 1,
\\
\rank \begin{pmatrix}
w_0 & \epsilon^3 & t_0 \\
z'  & x & \epsilon^3 \\
\end{pmatrix} \leq 1,
\rank \begin{pmatrix}
w_1 & \epsilon y & t_1 \\
z'  & x          & \epsilon^3 \\
\end{pmatrix} \leq 1,
\rank \begin{pmatrix}
v   & y^2 & t_2        & t_3 \\
z'  & x   & \epsilon^2 & y \\
\end{pmatrix} \leq 1,
\end{gather*}
and 
\begin{gather*}
\vectoriii
{s}
{w_0 \quad w_1 \quad z' \quad v}
{t_0 \quad t_1 \quad t_2 \quad \epsilon^3 \quad \epsilon t_2 \quad y\epsilon \quad \epsilon t_3 \quad x \quad y^2 \quad y t_3}
\in V_4 
 \\ 
:= \biggset{
\vectoriii{T_0^2}{T_0 T_1 \quad T_0 T_2 \quad T_0 T_3 \quad T_0 T_4}
{T_1^2 \quad T_1 T_2 \quad T_2^2 \quad T_1 T_3 \quad T_1 T_4 \quad T_2 T_3 \quad T_2 T_4 \quad T_3^2 \quad T_3 T_4 \quad T_4^2}},
\end{gather*}
where $V_4 \subset \bA^{15}$ is the cone over the image of the second Veronese map $\bP^4 \to \bP^{14}$.

%
We have 
$\cX_x \cap Z = (\varpi = t_0 = t_1 = t_2 = w_0 = w_1 = 0)$ and 
$\cX_x \cap E = (\varpi = x = z' = y = v = 0)$.


Let $Q_2, Q_3, Q_4$ be the points on $\cX_x \cap Z \cap E$ where $s$ ($ = 1 + t_3^3 + \dots$) vanish.
Around these points, $t_3$ is a unit.
and hence the following elements can be eliminated:
$x = t_3^{-1} y^3$, $t_1 = t_3^{-1} t_2^2$, $t_0 = t_3^{-2} t_2^3$, $w_0 = t_3^{-1} t_2 w_1$, $z' = t_3^{-1} y v$.
Thus the maximal ideals of the local rings are generated by $v, w_1, y, t_2, s, \varpi$, 
subject to $(s, w_1, v, t_2, \epsilon t_3, y t_3) \in V_2$,
where $V_2$ is (as in Section \ref{subsubsec:case 4,4}) the cone over the image of the Veronese map $\bP^2 \to \bP^5$.
It is straightforward to check that they are RDPs of type $A_1$ on both components  $Z = (w_1 = \varpi = t_2 = 0)$ and $E = (v = y = \varpi = 0)$,
and that the blow-up of $\cX'$ at $(s, v, y, w_1, \epsilon, t_2)$ is strictly semistable in the broad sense above these points.

Let $Q_1$ be the point $(t_i = x = y = z' = v = w_i = 0)$.
Around this point $s$ is a unit, hence the maximal ideal of the local ring is generated by $w_0,w_1,t_3,v,z',y,\varpi$,
subject to $(t_3, v, y s, z' s, w_1, \epsilon s, w_0 s) \in V'$, 
where \[ V' = \set{(T_0^3, T_0^2 T_1, T_0 T_1^2, T_1^3, T_0 T_2, T_1 T_2, T_2^3)} \subset \bA^7 \] is the toric variety 
attached to the monoid $M \subset \bN^3$ generated by the vectors \[ (3,0,0), (2,1,0), (1,2,0), (0,3,0), (1,0,1), (0,1,1), (0,0,3) .\]
(In other words, the elements $(t_3, \dots, w_0 s)$ define a morphism $k[M] \to \cO_{\cX', Q_1}$
from the monoid ring.)
We can check that, around this point, 
\begin{itemize}
\item the component $Z$ is $(\varpi = w_1 = w_0 = 0)$ and $\cO_{Z,Q_1}$ is $\frac{(1,1)}{3}$,
\item the component $E$ is $(\varpi = v = y = z' = 0)$ and $\cO_{E,Q_1}$ is an RDP of type $A_2$.
\end{itemize}
Let $\cX'' \to \cX'$ be the blow-up at $(t_3, v, y, z', w_1, \epsilon, w_0)$.
Then $\cX''$ is strictly semistable in the broad sense.
More precisely, there are $4$ components $\tilde{Z}, \tilde{E}, E_1, E_2$ over this point,
$\tilde{Z}$ and $\tilde{E}$ are the minimal resolutions of $Z$ and $E$ at $Q_1$ respectively,
$E_1 \cong \bP^2$ and $E_2 \cong \bF_2$ are exceptional divisors,
and the local equations at triple points are 
$\frac{v}{t_3} \frac{w_1}{t_3} t_3 - \epsilon s$ at $\tilde{E} \cap \tilde{Z} \cap E_2$
and $\frac{w_0 s}{w_1} \frac{t_3}{w_1} \frac{\epsilon s}{w_1} - \epsilon s$ at $\tilde{E} \cap E_1 \cap E_2$.


The assertion on $H^1(\cO_E)$ and $K_E$ follows also similarly
since $E$ is a hypersurface of degree $18 = 9 + 6 + 2 + 1$ in $\bP(9,6,2,1)$.

\subsection{End of Proof of Theorem \ref{thm:good reduction:bis}} \label{subsec:end of proof}

Let $\cA$ be the N\'eron model of the abelian surface $A$.
By Proposition \ref{prop:supersingular quotient}, 
the singularity of $\cA_k / \set{\pm 1}$ is an elliptic double point of type $\EDP_{4,4}^{(1)}$ or $\EDP_{2,8}^{(1)}$,
according to the supersingular abelian surface $\cA_k$ being superspecial or not.
We apply the normalized weighted blow-up $\cX'' \to \cX' \to \cX = \cA/\set{\pm 1}$ described in Sections \ref{subsec:blow-up}--\ref{subsec:explicit computation}.
Since $A[2]$ acts on itself (by translation) transitively,
all singularity of $\Sing(E) \setminus Z$ are of the same type.
From the classification of possible singularities under this assumption (Claim \ref{claim:description of Sing} (\ref{claim:item:44:same type}) and (\ref{claim:item:28:same type})),
either $\Sing(E) \setminus Z$ consists of RDPs,
or it is $1 \EDP_{2,8}^{(0)}$.
In the latter case, since we know that the affine surface $E \setminus Z$ is the quotient by $G'$, we can apply the blow-up construction once more.
Thus, in each case, we obtain a model $\cX''$ that is strictly semistable in the broad sense outside RDPs on the exceptional component.
In each case, moreover, again by Claim \ref{claim:description of Sing} (\ref{claim:item:44:same type}) and (\ref{claim:item:28:same type}),
the configuration of the RDPs on the special fiber is one of $16 A_1, 4 D_4^0, 2 D_8^0, 2 E_8^0, 1 D_{16}^0$.
We apply Proposition \ref{prop:flat blowup of RDP} (16 times) to $16 A_1$ on the generic fiber and obtain a proper model that is strictly semistable in the broad sense.

Applying Theorem \ref{thm:sufficient condition for good reduction}, we conclude that $X = \Km(A)$ has potential good reduction.

\section{Examples}

We give explicit examples of abelian surface for which $16 A_1$, $4 D_4^0$, $2 D_8^0$ occur as $\Sing(E) \setminus Z$.
Also we show that $\Sing(E) \setminus Z$ may differ between isogenous abelian surfaces.

Let $\cO$ be a discrete valuation ring as before (in particular $\charac K = 0$ and $\charac k = 2$).
The $\idealp$-adic valuation is denoted (additively) by $\map{v}{\cO}{\setQ_{\geq 0} \cup \set{\infty}}$.

Elliptic curves over $K$ having good reduction can be written, after replacing $\cO$, in the form 
$Y^ 2 + c X Y + Y = X^3$ with $c \in \cO$. 
Let us write this curve $E_c$.
The reduction of $E_c$ is defined by the same equation (with coefficients considered modulo $\idealp$),
and is supersingular if and only if $c \in \idealp$.
Using coordinate change $u = \frac{X}{Y}$ and $\overline{u} = -\frac{X}{Y + c X + 1}$,
we obtain the form $u + \overline{u} + c u \overline{u} - (u \overline{u})^2 = 0$,
with origin at $u = \overline{u} = 0$,
and the inversion map given by $u \leftrightarrow \overline{u}$.

Let $c_1, c_2 \in \idealp$ and 
consider $A_{c_1, c_2} := E_{c_1} \times E_{c_2}$,
$u + \overline{u} + c_1 u \overline{u} - (u \overline{u})^2 = 
 v + \overline{v} + c_2 v \overline{v} - (v \overline{v})^2 = 0$.
This abelian surface have good superspecial reduction.
Using the notations of Section \ref{subsubsec:case 4,4}, we have
$a = - c_1 x + x^2$ and $b = - c_2 y + y^2$.
Then $\epsilon$ is chosen so that $v(\epsilon) = \min\set{\frac{1}{2}v(c_1), \frac{1}{2}v(c_2), \frac{1}{3}v(2)}$.
We observe that $\Sing(E) \setminus Z$ is 
\begin{itemize}
\item $16 A_1  $ if $4 \epsilon^{-6} \in \cO^*$, equivalently if both $v(c_1)$ and $v(c_2)$ are $\geq \frac{2}{3}v(2)$,
\item $4  D_4^0$ if $4 \epsilon^{-6} \in \idealp$ and both $\epsilon^{-2} c_1 \in \cO^*$ and $\epsilon^{-2} c_2 \in \cO^*$, equivalently if $v(c_1) = v(c_2) < \frac{2}{3} v(2)$, 
\item $2  D_8^0$ if $4 \epsilon^{-6} \in \idealp$ and either $\epsilon^{-2} c_1 \in \idealp$ or $\epsilon^{-2} c_2 \in \idealp$, equivalently if $v(c_1)$ and $v(c_2)$ are different and at least one is  $< \frac{2}{3} v(2)$, 
\end{itemize}

Next, we will see that isogenous abelian varieties may result in different configuration of singularities on $E \setminus Z$.
Consider another elliptic curve $E_{c_1'}$ that admits an isogeny to $E_{c_1}$ of degree $2$
(if $c_1$ is generic, then there exist exactly $3$ such elliptic curves up to isomorphism).
Using the formula $j(E_c) = c^3 (c^3 - 24)^3 (c^3 - 27)^{-1}$ and the explicit form of the modular polynomial $\Phi_2(X, Y)$, 
we observe that 
\begin{itemize}
\item if $v(c_1) <    \frac{2}{3} v(2)$, then there exists such $c_1'$ with $v(c_1') = \frac{1}{2} v(c_1)$,
\item if $v(c_1) \geq \frac{2}{3} v(2)$, then there exists such $c_1'$ with $v(c_1') = \frac{1}{3} v(2)$.
\end{itemize}

%

We conclude that, while $A_{c_1, c_2}$ and $A_{c_1', c_2}$ are isogenous,
 the resulting configurations of singularities on $\Sing(E) \setminus Z$ differ.

\subsection*{Acknowledgments}
I thank 
Hiroyuki Ito, Kazuhiro Ito, Tetsushi Ito, Teruhisa Koshikawa, Ippei Nagamachi, Hisanori Ohashi, Teppei Takamatsu, and Fuetaro Yobuko 
for helpful comments and discussions.


\begin{bibdiv}
	\begin{biblist}
		\bib{Artin:wild2}{article}{
  author={Artin, M.},
  title={Wildly ramified $Z/2$ actions in dimension two},
  journal={Proc. Amer. Math. Soc.},
  volume={52},
  date={1975},
  pages={60--64},
  issn={0002-9939},
}

\bib{Katsura:Kummer2}{article}{
  author={Katsura, Toshiyuki},
  title={On Kummer surfaces in characteristic $2$},
  conference={ title={Proceedings of the International Symposium on Algebraic Geometry}, address={Kyoto Univ., Kyoto}, date={1977}, },
  book={ publisher={Kinokuniya Book Store, Tokyo}, },
  date={1978},
  pages={525--542},
}

\bib{Lazda--Skorobogatov:reductionofkummer}{article}{
  author={Lazda, Christopher},
  author={Skorobogatov, A. N.},
  title={Reduction of Kummer surfaces modulo 2 in the non-supersingular case},
  year={2022},
  eprint={https://arxiv.org/abs/2205.13831v1},
}

\bib{Liedtke--Matsumoto}{article}{
  author={Liedtke, Christian},
  author={Matsumoto, Yuya},
  title={Good reduction of K3 surfaces},
  journal={Compos. Math.},
  volume={154},
  date={2018},
  number={1},
  pages={1--35},
}

\bib{Matsumoto:SIP}{article}{
  author={Matsumoto, Yuya},
  title={On good reduction of some K3 surfaces related to abelian surfaces},
  journal={Tohoku Math. J. (2)},
  volume={67},
  date={2015},
  number={1},
  pages={83--104},
  issn={0040-8735},
}

\bib{Matsumoto:goodreductionK3}{article}{
  author={Matsumoto, Yuya},
  title={Good reduction criterion for K3 surfaces},
  journal={Math. Z.},
  volume={279},
  date={2015},
  number={1--2},
  pages={241--266},
  issn={0025-5874},
}

\bib{Matsumoto:k3alphap}{article}{
  author={Matsumoto, Yuya},
  title={$\mu _p$- and $\alpha _p$-actions on K3 surfaces in characteristic $p$},
  year={2022},
  journal={J. Algebraic Geom.},
  status={available electronically},
}

\bib{Overkamp:Kummer}{article}{
  author={Overkamp, Otto},
  title={Degeneration of Kummer surfaces},
  journal={Math. Proc. Cambridge Philos. Soc.},
  volume={171},
  date={2021},
  number={1},
  pages={65--97},
  issn={0305-0041},
}

\bib{Saito:logsmooth}{article}{
  author={Saito, Takeshi},
  title={Log smooth extension of a family of curves and semi-stable reduction},
  journal={J. Algebraic Geom.},
  volume={13},
  date={2004},
  number={2},
  pages={287--321},
  issn={1056-3911},
}

\bib{Serre--Tate}{article}{
  author={Serre, Jean-Pierre},
  author={Tate, John},
  title={Good reduction of abelian varieties},
  journal={Ann. of Math. (2)},
  volume={88},
  date={1968},
  pages={492--517},
  issn={0003-486X},
}

\bib{Takamatsu--Yoshikawa:mixed3fold}{article}{
  author={Takamatsu, Teppei},
  author={Yoshikawa, Shou},
  title={Minimal model program for semi-stable threefolds in mixed characteristic},
  year={2023},
  eprint={https://arxiv.org/abs/2012.07324v4},
}

\bib{Tate--Oort:groupschemes}{article}{
  author={Tate, John},
  author={Oort, Frans},
  title={Group schemes of prime order},
  journal={Ann. Sci. \'Ecole Norm. Sup. (4)},
  volume={3},
  date={1970},
  pages={1--21},
  issn={0012-9593},
}

\bib{Wagreich:ellipticsingularities}{article}{
  author={Wagreich, Philip},
  title={Elliptic singularities of surfaces},
  journal={Amer. J. Math.},
  volume={92},
  date={1970},
  pages={419--454},
  issn={0002-9327},
}


	\end{biblist}
\end{bibdiv}
\end{document}